\documentclass[12pt]{amsart}
\usepackage[active]{srcltx}
\usepackage{amscd,amssymb,amsopn,amsmath,amsthm,mathrsfs,graphics,amsfonts,enumerate,verbatim,calc
}
\usepackage{calc,amssymb,amsthm,amsmath,amscd, eucal,ulem}
\usepackage{alltt}
\usepackage[left=1.35in,top=1.25in,right=1.35in,bottom=1.25in]{geometry}
\RequirePackage[dvipsnames,usenames]{color}

\input{kmacros3.sty}
\input{xy}
\xyoption{all}
\usepackage{tikz}

\numberwithin{equation}{theorem}

\newcommand{\D}{\displaystyle}

\renewcommand{\m}{\mathfrak{m}}

\DeclareMathOperator{\depth}{depth}

\usepackage{fullpage}

\usepackage{setspace}
\usepackage{hyperref}

\usepackage{enumerate}

\usepackage{graphicx}

\usepackage[all,cmtip]{xy}
\usepackage{verbatim}

\theoremstyle{theorem}

\begin{document}
\title{$F$-injectivity and Buchsbaum singularities}
\author{Linquan Ma}
\address{Department of Mathematics\\ University of Michigan\\ Ann Arbor\\ Michigan 48109}
\email{lquanma@umich.edu}
\maketitle
\begin{abstract}
Let $(R,\m, K)$ be a local ring that contains a field. We show that, when $R$ has equal characteristic $p>0$ and when $H_\m^i(R)$ has finite length for all $i<\dim R$, then $R$ is $F$-injective if and only if every ideal generated by a system of parameters is Frobenius closed. As a corollary, we show that such an $R$ is in fact a Buchsbaum ring. This answers positively a question of S. Takagi that $F$-injective singularities with isolated non-Cohen-Macaulay locus are Buchsbaum. We also study the characteristic $0$ analogue of this question and we show that Du Bois singularities with isolated non-Cohen-Macaulay locus are Buchsbaum in the graded case.
\end{abstract}

\section{Introduction}

The concept of $F$-injective rings was first introduced in \cite{FedderFPureRationalsingularity} in the early 1980s. This class of rings naturally arises when one studies the Frobenius actions on the local cohomology supported at the maximal ideal, and is a natural generalization of $F$-pure rings. There is a notion of Frobenius closure for ideals (see Section 2 for detailed definition) that has close connections with $F$-pure and $F$-injective singularities. In general, the Frobenius closure is very small and is always contained in the tight closure (we refer to \cite{HochsterHunekeTC1} for basic tight closure theory). And, quite similar to the tight closure characterizations of $F$-regularity and $F$-rationality, it is well known that, under mild conditions on $R$, the $F$-purity of $R$ is the same as the condition that every ideal be Frobenius closed. If we assume $R$ is Cohen-Macaulay, then the $F$-injectivity of $R$ is equivalent to the condition that every ideal generated by a system of parameters be Frobenius closed, and also equivalent to the condition that a single ideal generated by a system of parameters be Frobenius closed.

However, when $R$ is not assumed to be Cohen-Macaulay, the relation between $F$-injectivity and Frobenius closure is not clear. The first goal of this paper is to explore the connections between these conditions without the Cohen-Macaulay assumption. Our first main theorem is the following:
\begin{theorem}
\label{main theorem}
Let $(R,\m)$ be a local ring of equal characteristic $p>0$. Suppose $H_\m^i(R)$ has finite length for each $i<\dim R$. Then the following are equivalent:
\begin{enumerate}
\item $R$ is $F$-injective.
\item Every ideal generated by a system of parameters is Frobenius closed.
\end{enumerate}
\end{theorem}

In general, $F$-injective singularities are not necessarily Cohen-Macaulay, but in many cases, they are Buchsbaum. This is a natural weakening of Cohen-Macaulayness and, in a precise sense, the closest condition to being Cohen-Macaulay (see Section 2 for details). In \cite{SchenzelApplicationsOfDualizingComplexes}, Schenzel proved some homological criteria for Bushsbaum singularities. And, utilizing the results of Hochster and Roberts on the purity of the Frobenius map in \cite{HochsterRobertsFrobeniusLocalCohomology}, Schenzel also obtained some sufficient conditions for Buchsbaum rings in the graded case. In fact, results in \cite{SchenzelApplicationsOfDualizingComplexes} indicate that for $(R,\m)$ an $F$-injective graded ring, if $H_\m^i(R)$ has finite length for each $i<\dim R$, then $R$ is Buchsbaum. We note that, under some mild conditions on $R$, $H_\m^i(R)$ has finite length for each $i<\dim R$ if and only if $R$ is Cohen-Macaulay on the punctured spectrum (we give a detailed explanation of this in Section 2). So Schenzel's result is basically saying that $F$-injective singularities with isolated non-Cohen-Macaulay locus are Buchsbaum in the graded case. Takagi asked whether the same conclusion holds when $(R,\m)$ is a local ring:
\begin{question}[{\it cf.} Open Problem A.3 in \cite{KovacsandSchwedesurveyonlogcanonicalandDuBoissingulaities}]
\label{Takagi's question}
Suppose $(R,\m)$ is $F$-injective and $H_\m^i(R)$ has finite length for each $i<\dim R$. Then is $R$ a Buchsbaum ring?
\end{question}

This is supported by results of Goto and Ogawa in \cite{GotoOgawaAnoteonringswithFLC} when $R$ is $F$-pure. Using Theorem \ref{main theorem}, we provide a positive answer to this question.
\begin{corollary}
\label{main corollary}
Let $(R,\m)$ be a local ring of equal characteristic $p>0$. Suppose $R$ is $F$-injective and $H_\m^i(R)$ has finite length for each $i<\dim R$. Then $R$ is a Buchsbaum ring.
\end{corollary}

It is known that $F$-injective singularities in characteristic $p>0$ have close connections with Du Bois singularities in characteristic $0$. This connection was studied intensively by Schwede in \cite{SchwedeFInjectiveAreDuBois}, where it was proved that in characteristic $0$, singularities of dense $F$-injective type are Du Bois and it was conjectured that the converse is also true. Based on this connection, it is quite natural to consider the characteristic $0$ analogue of Question \ref{Takagi's question}:
\begin{question}[{\it cf.} Open Problem A.4 in \cite{KovacsandSchwedesurveyonlogcanonicalandDuBoissingulaities}]
\label{Takagi's question in char 0}
Suppose $(R,\m)$ is Du Bois and $H_\m^i(R)$ has finite length for each $i<\dim R$. Then is $R$ a Buchsbaum ring?
\end{question}

Note that by results of Ishida in \cite{IshidaIsolatedDuBoissingularities}, Question \ref{Takagi's question in char 0} has a positive answer if $(R,\m)$ is a normal isolated singularity. Using Schenzel's criterion for Bushsbaum singularities in \cite{SchenzelApplicationsOfDualizingComplexes} and Schwede's simple characterization of Du Bois singularities in \cite{SchwedeEasyCharacterization}, we provide a positive answer to Question \ref{Takagi's question in char 0} in the normal standard graded case.
\begin{theorem}
\label{main theorem in characteristic 0}
Let $(R,\m)$ be a normal standard graded $K$-algebra (i.e., generated over $K$ by 1-forms). If $R$ is Du Bois and $H_\m^i(R)$ has finite length for each $i<\dim R$, then $R$ is a Buchsbaum ring.
\end{theorem}

Throughout this paper we will use $(R,\m)$ to denote either a Noetherian local ring with unique maximal ideal $\m$ or an $\mathbb{N}$-graded ring finitely generated over $K$ with unique homogeneous maximal ideal $\m$. Rings are always assumed to be equal characteristic (i.e., contain a field). In Section 2 we recall and review some basic definitions and properties about $F$-pure, $F$-injective and Buchsbaum rings. In Section 3 we work in equal characteristic $p>0$. We prove Theorem \ref{main theorem} and Corollary \ref{main corollary}. In Section 4 we work in equal characteristic $0$. We provide a criterion for Du Bois singularities for section rings of normal projective varieties and we prove Theorem \ref{main theorem in characteristic 0}, and we also remark that the conjecture that Du Bois singularities have dense $F$-injective type implies a positive answer to Question \ref{Takagi's question in char 0}.

\section{Preliminaries}

Let $(R,\m)$ be a local ring that contains a field. If $R$ has equal characteristic $p>0$, then there is a natural action of the Frobenius endomorphism of $R$ on each of its local cohomology modules $H_{\m}^i(R)$. Recall that a map of $R$-modules $N\rightarrow N'$ is {\it pure} if
for every $R$-module $M$ the map $N\otimes_RM\rightarrow N'\otimes_RM$ is injective. $R$ is called {\it $F$-pure} if the Frobenius endomorphism $R\xrightarrow{F} R$ is pure. $R$ is called {\it $F$-injective} if the Frobenius acts injectively on $H_\m^i(R)$ for every $i$. We point out that $F$-pure always implies $F$-injective (see \cite{HochsterRobertsFrobeniusLocalCohomology}).

When $R$ has equal characteristic $p>0$, for every ideal $I\subseteq R$, we define \[I^F=\{x\in R|\exists e, x^{p^e}\in I^{[p^e]} \}\] to be the {\it Frobenius closure} of $I$. $I$ is called {\it Frobenius closed} if $I^F=I$. It is well known that under mild conditions on $R$, $R$ is $F$-pure if and only if every ideal is Frobenius closed. Moreover, if $R$ is Cohen-Macaulay, then $R$ is $F$-injective if and only if every ideal generated by a system of parameters is Frobenius closed, and also if and only if one single ideal generated by a system of parameters is Frobenius closed.

We say that a local ring $(R,\m)$ has {\it finite local cohomology} if $H_\m^i(R)$ has finite length for every $i<\dim R$. It is well known that, under mild conditions on $R$, $R$ has finite local cohomology if and only if $R$ is equidimensional and Cohen-Macaulay on the punctured spectrum (see \cite{SchenzelTrungVerallgemeinerteCohenMacaulayModuln}). We will need the following important result characterizing rings with finite local cohomology. This result and its equivalent form appeared in \cite{SchenzelTrungVerallgemeinerteCohenMacaulayModuln}, \cite{Schenzelstandardsystemsofparemeters} and \cite{GotoOgawaAnoteonringswithFLC}. We recall that a sequence of elements $x_1,\dots,x_r$ in a local ring $R$ is called a {\it $d$-sequence} if $(x_1,\dots,x_{i-1}):x_ix_j=(x_1,\dots,x_{i-1}):x_j$ for every $1\leq i\leq j\leq r$.

\begin{theorem}[{\it cf.} Proposition 2.1 and 2.2 in \cite{Schenzelstandardsystemsofparemeters} and the main Theorem in \cite{GotoOgawaAnoteonringswithFLC}]
\label{FLCandd-sequence}
Let $(R,\m)$ be a local ring of dimension $n$, then the following are equivalent:
\begin{enumerate}
\item $H_{\m}^i(R)$ has finite length for all $i\neq n$.
\item There exists an integer $N$ such that for every system of parameters $x_1,\dots, x_n$ contained in $\m^N$, we have \[(x_1,\dots,x_{i-1}):x_i=(x_1,\dots,x_{i-1}):\m^N.\]
\item There exists an integer $N$ such that every system of parameters contained in $\m^N$ is a $d$-sequence.
\item There exists an integer $N$ and a constant $C$ such that for every system of parameters $x_1,\dots,x_n$ of $R$, we have \[l(R/(\underline{x}))-e(\underline{x}, R)\leq C\] with equality when $(x_1,\dots,x_n)\subseteq \m^N$.
\end{enumerate}
Moreover, when the equivalent conditions hold, we can let \[C=\sum_{i=1}^{n-1}\binom{n-1}{i}l(H_\m^i(R)).\]
\end{theorem}

Now we give the definition of Buchsbaum rings. It turns out that there are many different ways to define them. We also note that the definition of Buchsbaum rings is characteristic free (in fact it makes sense in mixed characteristic also, but we will not use this).
\begin{definition}\label{definition of Buchsbaum ring}
The following conditions on a local ring $(R,\m, K)$ are equivalent:
\begin{enumerate}
\item For every system of parameters $x_1,\dots,x_n$, we have \[(x_1,\dots,x_{i-1}):x_i=(x_1,\dots,x_{i-1}):\m\]for every $i$.
\item Every system of parameters is a $d$-sequence.
\item The difference $l_R(R/J)-e(J,R)$, where $J$ is an ideal generated by a system of parameters, is an invariant of $R$ (i.e., it is independent of $J$).
\item There is a system of parameters $\underline{x}=x_1,\dots,x_n$ such that $\tau^nC^{\bullet}(\underline{x}, R)$ is quasi-isomorphic to a complex of $K$-vector spaces, where $\tau^nC^{\bullet}(\underline{x}, R)$ is the \v{C}ech complex truncated from above at the $n$-th place.
\end{enumerate}
When $R$ satisfies one of these equivalent conditions, it is called a {\it Buchsbaum} ring.
\end{definition}

We refer to \cite{HunekeTheoryofdSequence} for $(1)\Leftrightarrow(2)$, \cite{StuckradandVogelEineVerallgemeinerungderCohenMacaulayRinge} for $(1)\Leftrightarrow(3)$ and \cite{SchenzelApplicationsOfDualizingComplexes} for $(1)\Leftrightarrow(4)$. Next we summarize some basic facts about Buchsbaum rings.
\begin{remark}
\begin{enumerate}
\item Cohen-Macaulay rings are obviously Buchsbaum. Moreover, by Definition \ref{definition of Buchsbaum ring}, it is clear that $R$ is Buchsbaum if and only if we can take $N=1$ in (2)-(4) in Theorem \ref{FLCandd-sequence}. So among rings with finite local cohomology, Buchsbaum rings are the closest to Cohen-Macaulay rings.
\item By $(4)$ in Definition \ref{definition of Buchsbaum ring}, $R$ is Buchsbaum implies $H_\m^i(R)$ are $K$-vector spaces for all $i<n=\dim R$. However, there exist local rings such that $H_\m^i(R)$ are $K$-vector spaces for all $i<n=\dim R$ but $R$ is {\it not} Buchsbaum (see \cite{GotonoteonquasiBuchsbaumrings}).
\end{enumerate}
\end{remark}

We also mention the notion of Buchsbaum modules introduced in \cite{StuckradandVogelTowardatheoryofBuchsbaumsingularities}. A finitely generated module $M$ over a local ring $(R,\m)$ is called a {\it Buchsbaum module} of dimension $d$ if \[(x_1,\dots,x_{i-1})M:_Mx_i=(x_1,\dots,x_{i-1})M:_M\m\] for every system of parameters $(x_1,\dots,x_d)$ of $M$ and every $1\leq i\leq d$. So $R$ is Buchsbaum if and only if $R$ is a Buchsbaum module over $R$ of maximal dimension $=\dim R$. We will need the following powerful criterion, the so-called surjectivity criterion, of Buchsbaum modules:
\begin{theorem}[{\it cf.} Theorem 1 in \cite{StuckradandVogelTowardatheoryofBuchsbaumsingularities} and Satz 2 in \cite{StuckradsurjectivitycriterionofBuchsbaum}]
\label{surjectivity criterion for Buchsbaum modules}
Let $M$ be a finitely generated module over a local ring $(R,\m, K)$. If the canonical maps $\Ext_R^i(K, M)\to H_\m^i(M)$ are surjective for all $i\neq d=\dim M$ then $M$ is a Buchsbaum module. Moreover, if $R$ is regular, then the converse also holds.
\end{theorem}

In \cite{SchenzelApplicationsOfDualizingComplexes}, Schenzel observed the following criterion for rings to be Buchsbaum in the graded case which turns out to be very useful. In fact, this follows easily from (4) in Definition \ref{definition of Buchsbaum ring}.
\begin{theorem}[{\it cf.} Theorem 3.1 in \cite{SchenzelApplicationsOfDualizingComplexes}]
\label{sufficient condision for Buchsbaum ring in the graded case}
Let $(R,\m)$ be a special graded $K$ algebra (meaning that $R$ is non-negatively graded of finite type over $K$). If there exists an integer $t$ such that $[H_\m^i(R)]_s=0$ for all $s\neq t$ and for every $i<\dim R$, then $R$ is a Buchsbaum ring.
\end{theorem}

When $(R,\m)$ is special graded, $F$-injective and has finite local cohomology, it is easy to see that for every $i<\dim R$, $H_\m^i(R)=[H_\m^i(R)]_0$. Hence in this case, we can apply Theorem \ref{sufficient condision for Buchsbaum ring in the graded case} with $t=0$. Therefore if $(R,\m)$ is special graded, $F$-injective and $R$ has finite local cohomology, then $R$ is Buchsbaum. This fact is known to experts. In fact, this is exactly Proposition 4.1 in \cite{SchenzelApplicationsOfDualizingComplexes}. Although Schenzel requires that $R$ be $F$-pure, exactly the same argument works when $R$ is $F$-injective.

\section{Characteristic $p>0$ results}
In this section we prove our main results in characteristic $p>0$, Theorem \ref{main theorem} and Corollary \ref{main corollary}. Throughout this section all rings are of equal characteristic $p>0$ (although we will repeat this sometimes). We start by proving two simple lemmas that we will use.
\begin{lemma}
\label{lemma on ideals generated by part of a system of parameters}
If every ideal generated by a full system of parameters is Frobenius closed, then so is every ideal generated by part of a system of parameters.
\end{lemma}
\begin{proof}
Suppose $(x_1,\dots,x_t)$ is part of a system of parameters, contained in $(x_1,\dots,x_t,x_{t+1},\dots,x_n)$. If $y\in (x_1,\dots,x_t)^F$, then $y\in (x_1,\dots,x_t,x_{t+1}^s,\dots,x_n^s)^F=(x_1,\dots,x_t,x_{t+1}^s,\dots,x_n^s)$ for every $s>0$. So \[y\in \bigcap_s(x_1,\dots,x_t,x_{t+1}^s,\dots,x_n^s)=(x_1,\dots,x_t).\]
\end{proof}

\begin{lemma}
\label{lemma on Frobenius action on top local cohomology}
Let $(R,\m)$ be a local ring such that every ideal generated by a system of parameters is Frobenius closed. Then the Frobenius acts injectively on $H_\m^j(R)$ for $j=\dim R$ and $j=\depth R$.
\end{lemma}
\begin{proof}
First notice that if $J^{[q]}=(x_1^q,\dots,x_j^q)$ is Frobenius closed for every $q=p^e$, then Frobenius acts injectively on $H_J^j(R)$. This is because we have a direct limit system
\[  \xymatrix{
    R/J \ar[r]  \ar[d]^{F} &  R/J^{[p]}  \ar[r]  \ar[d]^{F}   &  R/J^{[p^2]} \ar[r]\ar[d]^{F} &  \cdots \\
    R/J^{[p]}\ar[r]  &  R/J^{[p^2]} \ar[r]   &  R/J^{[p^3]} \ar[r] & \cdots 
} \]
where the vertical maps are the Frobenius and the horizontal maps are multiplications by $(x_1\cdots x_j)^{p^e-p^{e-1}}$ at the corresponding spots. The direct limit of both lines are $H_J^j(R)$ and the vertical map is exactly the Frobenius action on $H_J^j(R)$. Since $J^{[q]}$ is Frobenius closed, we know that each vertical map is injective, hence so is the induced map on the direct limit. Therefore Frobenius acts injectively on $H_J^j(R)$. From this it follows immediately that if every ideal generated by a system of parameters is Frobenius closed, then Frobenius acts injectively on $H_\m^n(R)$ for $n=\dim R$.

Now we let $I=(x_1,\dots,x_n)$ be a system of parameters with $(x_1,\dots,x_r)$ a maximal regular sequence in $R$ (i.e., $r=\depth R$). By Lemma \ref{lemma on ideals generated by part of a system of parameters}, we know that $(x_1^q,\dots,x_r^q)$ is Frobenius closed for every $q=p^e$. We have the local cohomology spectral sequence: \[E_2^{p,q}=H_{I}^p(H_J^q(R))\Rightarrow H_{I+J}^{p+q}(R).\] We apply this spectral sequence to $I=\m$ and $J=(x_1,\dots,x_r)$. Since $(x_1,\dots,x_r)$ is a regular sequence, $H_J^q(R)$ vanishes unless $q=r$. So this spectral sequence degenerates. So we know that \[H_\m^0(H_J^r(R))\cong H_\m^r(R).\] But since $J^{[q]}=(x_1^q,\dots,x_r^q)$ is Frobenius closed for every $q=p^e$, the above argument shows that Frobenius acts injectively on $H_J^r(R)$, hence it also acts injectively on $H_\m^0(H_J^r(R))\cong H_\m^r(R)$.
\end{proof}

The next proposition is essentially taken from \cite{GotoOgawaAnoteonringswithFLC}, where the authors show that for rings of finite local cohomology, $F$-purity implies Buchsbaumness. But in fact, the argument in \cite{GotoOgawaAnoteonringswithFLC} only uses that every ideal generated by a system of parameters is Frobenius closed. We give a short proof of this proposition for completeness.
\begin{proposition}[{\it cf.} main Corollary in \cite{GotoOgawaAnoteonringswithFLC}]
\label{s.o.p Frobenius closed implies Buchsbaum}
Let $(R,\m)$ be a local ring of equal characteristic $p>0$. Suppose $R$ has finite local cohomology and every ideal generated by a system of parameters is Frobenius closed. Then $R$ is Buchsbaum.
\end{proposition}
\begin{proof}
We claim that every system of parameters is a $d$-sequence. Since $H_\m^i(R)$ has finite local cohomology, by Theorem \ref{FLCandd-sequence} $(1)\Rightarrow(2)$, there exists $N$ such that every system of parameters contained in $\m^N$ is a $d$-sequence. Let $x_1,\dots,x_n$ be an arbitrary system of parameters. Note that by Lemma \ref{lemma on ideals generated by part of a system of parameters}, $(x_1,\dots,x_i)$ is Frobenius closed for every $i$. We want to show $(x_1,\dots,x_{i-1}):x_ix_j=(x_1,\dots,x_{i-1}):x_j$. One containment is obvious. For the other one, let $y\in(x_1,\dots,x_{i-1}):x_ix_j$, for $q\geq N$, we have
\begin{eqnarray*}
&&yx_ix_j\in(x_1,\dots,x_{i-1})\\
&\Rightarrow&y^qx_i^qx_j^q\in(x_1^q,\dots,x_{i-1}^q)\\
&\Rightarrow&y^qx_j^q\in(x_1^q,\dots,x_{i-1}^q)\\
&\Rightarrow&yx_j\in(x_1,\dots,x_{i-1})\\
&\Rightarrow&y\in(x_1,\dots,x_{i-1}):x_j
\end{eqnarray*}
where the third line we use $(x_1^q,\dots,x_n^q)$ is a $d$-sequence (because $q\geq N$), the fourth line we use that $(x_1,\dots,x_{i-1})$ is Frobenius closed.
\end{proof}


We briefly review the $\Gamma$-construction introduced in \cite{HochsterHunekeFRegularityTestElementsBaseChange}. Let $K$ be a field of positive characteristic $p>0$ with a $p$-base $\Lambda$. Let $\Gamma$ be a fixed cofinite subset of $\Lambda$. For $e\in \mathbb{N}$ we denote by $K^{\Gamma,e}$ the purely inseparable field extension of $K$ that is the result of adjoining $p^e$-th roots of all elements in $\Gamma$ to $K$. Now for $(R,\m)$ a complete local ring with $K\subseteq R$ a coefficient field, let $x_1,\dots,x_n$ be a system of parameters for $R$. We know that $R$ is module-finite over $A=K[[x_1,\dots,x_n]]\subseteq R$. Let $A^\Gamma$ denote $\bigcup_{e\in\mathbb{N}}K^{\Gamma,e}[[x_1,\dots,x_n]]$, which is a regular local ring that is faithfully flat and purely inseparable over $A$. The maximal ideal of $A$ expands to that of $A^{\Gamma}$. Let $R^{\Gamma}=A^{\Gamma}\otimes_AR$, which is module-finite over the regular ring $A^{\Gamma}$ and is faithfully flat and purely inseparable over $R$. The maximal ideal of $R$ expands to the maximal ideal of $R^{\Gamma}$ and the residue field of $R^{\Gamma}$ is $K^{\Gamma}=\bigcup_{e\in\mathbb{N}}K^{\Gamma,e}$.

We will use the important fact that $R^\Gamma$ is $F$-finite (see \cite{HochsterHunekeFRegularityTestElementsBaseChange} for details). Moreover, we can preserve many good properties of $R$ when $\Gamma$ is sufficiently small. For example, if $(R,\m)$ is complete and $F$-injective, then $R^\Gamma$ is still $F$-injective for any sufficiently small choice of cofinite $\Gamma$ by Lemma 2.9 in \cite{EnescuHochsterTheFrobeniusStructureOfLocalCohomology}.

\vspace{1em}

Now we state and prove our main result in characteristic $p>0$.

\begin{theorem}
\label{main theorem 2}
Let $(R,\m)$ be a local ring of equal characteristic $p>0$ and dimension $n$. Suppose $R$ has finite local cohomology. Then the following are equivalent:
\begin{enumerate}
\item $R$ is $F$-injective.
\item Every ideal generated by a system of parameters is Frobenius closed.
\end{enumerate}
\end{theorem}
\begin{proof}
We first prove $(1)\Rightarrow(2)$. Since $R$ is $F$-injective, so is $\widehat{R}$. We apply the $\Gamma$-construction to $\widehat{R}$: $\widehat{R}^\Gamma$ is $F$-finite and $F$-injective for any sufficiently small choice of cofinite $\Gamma$. Now we consider $S=\widehat{\widehat{R}^\Gamma}$, we know that $S$ is complete, $F$-finite, $F$-injective and faithfully flat over $R$. If we can show that every ideal generated by a system of parameters in $S$ is Frobenius closed, then the same follows for $R$ because for every $I\subseteq R$ generated by a system of parameters, we have $I^F\subseteq (IS)^F\bigcap R=IS\bigcap R=I$. Therefore, to prove $(1)\Rightarrow(2)$, we may replace $R$ by $S$ and hence assume that $R$ is complete and $F$-finite. Note that in this case since $R$ is $F$-injective, we know $R$ is reduced (for example, see Remark 2.6 in \cite{SchwedeandZhangBertinitheoremsforFsingularities}), hence we may also assume without loss of generality that $R$ is reduced.

Let $R^{1/q}$ denote the ring obtained by adjoining all $q$-th roots of elements of $R$ where $q=p^e$. Since $R$ is reduced, we have a short exact sequence \[0\rightarrow R\rightarrow R^{1/q}\rightarrow R^{1/q}/R\rightarrow 0\] which induces a long exact sequence of local cohomology
\begin{equation}
\label{long exact sequence of local cohomology}
\cdots\rightarrow H_\m^{i-1}(R^{1/q}/R)\xrightarrow{\phi_i} H_\m^i(R)\rightarrow H_\m^i(R^{1/q})\rightarrow H_\m^i(R^{1/q}/R)\rightarrow\cdots.
\end{equation}
Because $R$ is $F$-injective, each $H_\m^i(R)\rightarrow H_\m^i(R^{1/q})$ is injective. This means each connecting map $\phi_i$ is the zero map. So (\ref{long exact sequence of local cohomology}) actually gives us $n-1$ short exact sequences:
\begin{equation}
\label{short exact sequence of local cohomology}
0 \rightarrow H_\m^i(R)\rightarrow H_\m^i(R^{1/q})\rightarrow H_\m^i(R^{1/q}/R)\rightarrow0
\end{equation}
for every $0\leq i\leq n-1$.

Let $x_1,\dots,x_n$ be any system of parameters, we want to show that $(x_1,\dots,x_n)$ is Frobenius closed. Since $R$ is complete, we may pick a coefficient field $K\cong R/\m$ of $R$, and by Cohen's structure theorem, $R$ is module finite over $A=K[[x_1,\dots,x_n]]$. Hence $R^{1/q}$ is also module finite over $A$ for every $q=p^e$ since $R$ is $F$-finite. By Theorem \ref{FLCandd-sequence} $(1)\Rightarrow(2)$, if $q>N$, every system of parameters in $R$ satisfies \[(x_1^q,\dots,x_{i-1}^q):_R\m^{[q]}\subseteq(x_1^q,\dots,x_{i-1}^q):_Rx_i^q=(x_1^q,\dots,x_{i-1}^q):_R\m^N\subseteq(x_1^q,\dots,x_{i-1}^q):_R\m^{[q]}.\] So we must have equalities. But after taking $q$-th roots, this implies
\[(x_1,\dots,x_{i-1})R^{1/q}:_{R^{1/q}}x_i=(x_1,\dots,x_{i-1})R^{1/q}:_{R^{1/q}}\m.\] In particular, this implies that when $q>N$, $R^{1/q}$ is a (finitely generated) Buchsbaum $R$-module of dimension $n$, and hence also a Buchsbaum $A$-module of dimension $n$.

Now we claim that for every $q>N$, $R^{1/q}/R$ is also a Buchsbaum module of dimension $n$ over $A$. We prove this using the surjectivity criterion of Buchsbaum modules (Theorem \ref{surjectivity criterion for Buchsbaum modules}). We have the following commutative diagram, which are the long exact sequences of $\Ext_A^i(K,-)$ and $H_\m^i(-)$ induced by $0\to R\to R^{1/q}\to R^{1/q}/R\to 0$:
\[  \xymatrix{
    \cdots \ar[r]  & \Ext_A^i(K, R) \ar[r] \ar[d] &  \Ext_A^i(K, R^{1/q})  \ar[r]  \ar[d]^{\alpha_i}   &  \Ext_A^i(K, R^{1/q}/R)  \ar[r] \ar[d]^{\beta_i} &  \cdots \\
    0\ar[r]  & H_\m^i(R) \ar[r]   &  H_\m^i(R^{1/q}) \ar[r] & H_\m^i(R^{1/q}/R) \ar[r] & 0 
} \]
where the bottom sequence is exact by (\ref{short exact sequence of local cohomology}). Since $R^{1/q}$ is a Buchsbaum $A$-module of dimension $n$ and $A$ is a regular local ring, we know that for each $0\leq i\leq n-1$, $\alpha_i$ is surjective. So by the commutativity of the above diagram, each $\beta_i$ is also surjective. Hence $R^{1/q}/R$ is a Buchsbaum $A$-module of dimension $n$ for every $q>N$.

Now we apply Theorem \ref{FLCandd-sequence} $(1)\Rightarrow(4)$ for $(\underline{x})=(x_1,\dots,x_n)$, we have
\begin{equation}
\label{3.5.3}
l(R/(\underline{x}))-e(\underline{x}, R)\leq \sum_{i=1}^{n-1}\binom{n-1}{i}l(H_\m^i(R))
\end{equation}
Since $R^{1/q}$ and $R^{1/q}/R$ are Buchsbaum modules over $A$, we know from Bemerkung (4.2) in \cite{SchenzelTrungVerallgemeinerteCohenMacaulayModuln} that
\begin{equation}
\label{3.5.4}
l(R^{1/q}/(\underline{x})R^{1/q})-e(\underline{x}, R^{1/q})=\sum_{i=1}^{n-1}\binom{n-1}{i}l(H_\m^i(R^{1/q}))
\end{equation}
and
\begin{equation}
\label{3.5.5}
l(\D\frac{R^{1/q}/R}{(\underline{x})(R^{1/q}/R)})-e(\underline{x}, R^{1/q}/R)=\sum_{i=1}^{n-1}\binom{n-1}{i}l(H_\m^i(R^{1/q}/R))
\end{equation}
where the length $l(-)$ and multiplicity $e(-)$ are considered as length and multiplicity computed over $A$, which are the same as the length and multiplicity computed over $R$ since $R$ is module finite over $A$ with the same residue field. Now we consider (\ref{3.5.3})-(\ref{3.5.4})+(\ref{3.5.5}). The left hand side is just \[l(R/(\underline{x}))-l(R^{1/q}/(\underline{x})R^{1/q})+l(\D\frac{R^{1/q}/R}{(\underline{x})(R^{1/q}/R)})\] because the multiplicities cancel. The right hand side is zero because of (\ref{short exact sequence of local cohomology}). Hence we know that
\begin{equation}
\label{3.5.6}
l(R/(\underline{x}))+l(\D\frac{R^{1/q}/R}{(\underline{x})(R^{1/q}/R)})\leq l(R^{1/q}/(\underline{x})R^{1/q}).
\end{equation}

On the other hand, we can also apply $\otimes_RR/(\underline{x})$ to the short exact sequence \[0\to R\to R^{1/q}\to R^{1/q}/R\to 0,\] and we get \[\Tor_1^R(R/(\underline{x}), R^{1/q}/R)\xrightarrow{\varphi} R/(\underline{x})\rightarrow R^{1/q}/(\underline{x})R^{1/q}\rightarrow \D\frac{R^{1/q}/R}{(\underline{x})(R^{1/q}/R)}\rightarrow 0.\] So (\ref{3.5.6}) implies that $\varphi$ must be the zero map. Hence for every $q>N$, we have an injection $0\to R/(\underline{x})\rightarrow R^{1/q}/(\underline{x})R^{1/q}$. But this map is the same as the Frobenius map: $0\to R/(\underline{x})\rightarrow R/(\underline{x^q})$. Now if $y\in (x_1,\dots,x_n)^F$, then $y^q\in (x_1^q,\dots,x_n^q)$ for some $q>N$, so $\overline{y}$ maps to $0$ under $0\to R/(\underline{x})\rightarrow R/(\underline{x^q})$, hence $y\in(x_1,\dots,x_n)$. This proves that every ideal generated by a system of parameters is Frobenius closed.

\vspace{1em}


Now we prove $(2)\Rightarrow(1)$. Since every ideal generated by a system of parameters is Frobenius closed, by Proposition \ref{s.o.p Frobenius closed implies Buchsbaum} we know that $R$ is Buchsbaum. But since $R$ is Buchsbaum, we know that every system of parameters is a {\it standard} system of parameters in the sense of Schenzel \cite{Schenzelstandardsystemsofparemeters} (see Corollary 3.6 in \cite{Schenzelstandardsystemsofparemeters}). Hence by Proposition 3.3 in \cite{Schenzelstandardsystemsofparemeters}, we know that for every ideal $I=(x_1,\dots,x_n)$ generated by a system of parameters in $R$ and every $0\leq i\leq n-1$, there are natural isomorphisms:
\begin{equation}
\label{3.6.7}
H_\m^i(R)\cong \D\frac{(x_1,\dots,x_i):I}{(x_1,\dots,x_i)+\sum_{j=1}^i(x_1,\dots,\widehat{x_j},\dots,x_i):I}
\end{equation}
Now we observe that one can view the Frobenius map on $H_\m^i(R)$ as the natural map $H_\m^i(R)\to H_\m^i(R^{1/p})$ and then identify $R^{1/p}$ with $R$. It is straightforward to check that under (\ref{3.6.7}), the Frobenius action on $H_\m^i(R)$ is the same as the map
\[\frac{(x_1,\dots,x_i):I}{(x_1,\dots,x_i)+\sum_{j=1}^i(x_1,\dots,\widehat{x_j},\dots,x_i):I}\to \frac{(x_1^p,\dots,x_i^p):I^{[p]}}{(x_1^p,\dots,x_i^p)+\sum_{j=1}^i(x_1^p,\dots,\widehat{x_j^p},\dots,x_i^p):I^{[p]}}\] sending $\overline{y}$ to $\overline{y^p}$.

So in order to show that Frobenius acts injectively on $H_\m^i(R)$ for $1\leq i\leq n-1$, it suffices to show that if $y^p\in (x_1^p,\dots,x_i^p)+\sum_{j=1}^i(x_1^p,\dots,\widehat{x_j^p},\dots,x_i^p):I^{[p]}$, then $y\in (x_1,\dots,x_i)+\sum_{j=1}^i(x_1,\dots,\widehat{x_j},\dots,x_i):I$. But since $R$ is Buchsbaum, we know that $(x_1,\dots,x_i):I=(x_1,\dots,x_i):\m$ is the unmixed component of $(x_1,\dots,x_i)$ in its primary decomposition (we refer to \cite{GotoontheassociatedgradedringsofparameteridealsinBuchsbaumrings}, page 502-503 for a more detailed explanation of this). Now by Theorem 4.7 in \cite{GotoontheassociatedgradedringsofparameteridealsinBuchsbaumrings}, for every $1\leq i\leq n-1$ and every fixed $k\geq 2$, we have
\begin{equation}
\label{3.6.8}
(x_1,\dots,x_i)+\sum_{j=1}^i(x_1,\dots,\widehat{x_j},\dots,x_i):I=(x_1^k,\dots,x_i^k):(x_1x_2\cdots x_i)^{k-1}.
\end{equation}
So for every $1\leq i\leq n-1$, we have
\begin{eqnarray*}
&&y^p\in (x_1^p,\dots,x_i^p)+\sum_{j=1}^i(x_1^p,\dots,\widehat{x_j^p},\dots,x_i^p):I^{[p]}\\
&\Rightarrow&y^p(x_1^px_2^p\cdots x_i^p)^{k-1}\in(x_1^{pk},\dots,x_i^{pk})\\
&\Rightarrow&y(x_1x_2\cdots x_i)^{k-1}\in(x_1^k,\dots,x_i^k)\\
&\Rightarrow&y\in (x_1^k,\dots,x_i^k):(x_1x_2\cdots x_i)^{k-1}\\
&\Rightarrow&y\in (x_1,\dots,x_i)+\sum_{j=1}^i(x_1,\dots,\widehat{x_j},\dots,x_i):I\\
\end{eqnarray*}
where the last implication is by (\ref{3.6.8}), and the second implication we use the fact that $(x_1^k,\dots,x_i^k)$ is Frobenius closed (by Lemma \ref{lemma on ideals generated by part of a system of parameters}). Hence considering (\ref{3.6.7}), we have already showed that Frobenius acts injectively on each $H_\m^i(R)$ when $1\leq i\leq n-1$.

It remains to show that the Frobenius acts injectively on $H_\m^0(R)$ and $H_\m^n(R)$. But since every ideal generated by a system of parameters is Frobenius closed, we know that $R$ is reduced by the same argument as in the proof of Lemma \ref{lemma on ideals generated by part of a system of parameters}. So we know that $\depth R\geq 1$ and hence $H_\m^0(R)=0$. Furthermore Frobenius acts injectively on $H_\m^n(R)$ by Lemma \ref{lemma on Frobenius action on top local cohomology}. This completes the proof of $(2)\Rightarrow(1)$.
\end{proof}

Finally we can give a positive answer to Question \ref{Takagi's question}.

\begin{corollary}
\label{main corollary 2}
Let $(R,\m)$ be a local ring of equal characteristic $p>0$. Suppose $R$ is $F$-injective and $R$ has finite local cohomology. Then $R$ is Buchsbaum.
\end{corollary}
\begin{proof}
This follows immediately from Proposition \ref{s.o.p Frobenius closed implies Buchsbaum} and Theorem \ref{main theorem 2}.
\end{proof}

\begin{remark}
It is quite natural to ask whether $F$-injectivity is always equivalent to the assertion that every ideal generated by a system of parameters is Frobenius closed. We don't have a counter example yet.
\end{remark}

\section{Characteristic $0$ results}
In this section we study Question \ref{Takagi's question in char 0}, and we provide a positive answer when $R$ is a section ring of a normal projective variety, hence in particular we answer this question when $R$ is normal and standard graded (we say $R$ is standard graded if $R_0=K$ and $R$ is generated over $R_0$ by $R_1$). Throughout this section all rings and schemes are of finite type over a field $K$ of characteristic $0$. All schemes are separated. We first recall the definition of (strong) log resolutions. Let $X$ be a closed subscheme of $Y$ with ideal sheaf $\mathscr{I}$. A morphism $\pi$: $\widetilde{Y}\rightarrow Y$ is called a {\it log resolution} of the pair $(Y, X)$ if
\begin{enumerate}
\item $\pi$ is proper and birational with $\widetilde{Y}$ smooth
\item $\mathscr{I} O_{\widetilde{Y}}=O_{\widetilde{Y}}(-G)$ is an invertible sheaf corresponding to a divisor $-G$
\item $\Supp(G)\bigcup E$ has simple normal crossings where $E$ is the exceptional set of $\pi$.
\end{enumerate}
When $X=\emptyset$, we simply say $\pi$ is a log resolution of $Y$. We say $\pi$ is a {\it strong log resolution} of the pair $(Y, X)$ if moreover $\pi$ is an isomorphism outside of $X$. We note that when the characteristic of $K$ is $0$, log resolutions always exist and strong log resolutions exist if $Y$ is smooth (see \cite{HironakaResolution}).

We now recall Schwede's characterization of Du Bois singularities, which was shown to be equivalent to the classical definition using Hodge theoretic methods in \cite{SchwedeEasyCharacterization}.
\begin{definition}[{\it cf.} Theorem 4.6 in \cite{SchwedeEasyCharacterization}]
\label{Schwede's simple characterization of Du Bois}
Let $X\hookrightarrow Y$ be a reduced closed subscheme of a smooth scheme $Y$. Let $\pi$: $\widetilde{Y}\rightarrow Y$ be a strong log resolution of $(Y, X)$ and let $E$ be the reduced pre-image of $X$ in $\widetilde{Y}$. Then $X$ has {\it Du Bois} singularities if and only if the natural map $O_X\rightarrow \mathbf{R}\pi_{*}O_E$ is a quasi-isomorphism.
\end{definition}

Moreover, when $X$ is a Cohen-Macaulay normal scheme, there is another simple criterion for Du Bois singularities proved in \cite{KovacsSchwedeSmithLCImpliesDuBois} which is also useful.
\begin{theorem}[{\it cf.} Theorem 3.1 in \cite{KovacsSchwedeSmithLCImpliesDuBois}]
\label{canonicl sheaf of Du Bois singularity}
Suppose that $X$ is normal and Cohen-Macaulay. Let $\pi$: $X'\rightarrow X$ be any log resolution of $X$, and denote the reduced exceptional divisor of $\pi$ by $G$. Then $X$ has Du Bois singularities if and only if $\pi_*\omega_{X'}(G)\cong \omega_X$.
\end{theorem}

Now we are ready to prove our main result in characteristic $0$. It follows from some more general results. The first one is a Kodaira vanishing result for normal Cohen-Macaulay Du Bois singularities. This is well known to experts, and follows from a more general (and harder) result of Patakfalvi (Theorem 1.3 of \cite{PatakfalviSeminegativityDuBois}). But we also give a short, different proof for completeness. In fact the result follows easily from one of Fujino's vanishing theorems.
\begin{theorem}
\label{Kodaira vanishing for CM Du Bois}
Let $X$ be a normal projective scheme which is Cohen-Macaulay and Du Bois. Then $H^i(X, \omega_X\otimes\mathscr{L})=0$ for every ample line bundle $\mathscr{L}$ and every $i>0$.
\end{theorem}
\begin{proof}
One of Fujino's vanishing results (for example, see Theorem 1.1 in \cite{FujinoInjectivityvanishingandtorsionfreetheorems}) says that if $f$: $Y\rightarrow Z$ is projective with $Y$ smooth, and $B$ is an effective $\mathbb{Q}$-divisor with coefficients $\leq 1$ with simple normal crossing support, then we have \[H^p(Z, R^qf_*O_Y(K_Y+B+H))=0\] for every $p>0$, $q\geq 0$ if $H=f^*H'$ for some ample $H'$.

Take a log resolution $\pi$: $X'\rightarrow X$ of $X$ with reduced exceptional divisor $G$. In particular $G$ has coefficient $\leq 1$ with simple normal crossing support. Now we apply the above vanishing result to $Y=X'$, $Z=X$, $B=G$, $H=\pi^*\mathscr{L}$, $p=i>0$ and $q=0$. We get
\begin{equation}\label{3}
H^i(X, \pi_*(O_{X'}(K_{X'}+G)\otimes \pi^*\mathscr{L}))=0.
\end{equation}
Since we know that $\pi_*\omega_{X'}(G)\cong\omega_X$ when $X$ is normal Cohen-Macaulay and Du Bois by Theorem
\ref{canonicl sheaf of Du Bois singularity}. Applying the projection formula to (\ref{3}) we get $H^i(X, \omega_X\otimes\mathscr{L})=0.$
\end{proof}

Another key ingredient is the following result which gives a characterization of Du Bois singularities for section rings of ample line bundles. The proof makes use of the ``natural" construction from \cite{EGA}.
\begin{theorem}
\label{characterization of Du Bois in graded case}
Let $Z$ be a normal projective variety over $K$ and let $\mathscr{L}$ be an ample line bundle on $Z$. Let $R=\oplus_{i\in\mathbb{N}} H^0(Z,\mathscr{L}^i)$ be the section ring of $Z$ with respect to $\mathscr{L}$ and $\m$ be the irrelevant maximal ideal of $R$. Then the following are equivalent:
\begin{enumerate}
\item $R$ is Du Bois.
\item $Z$ has Du Bois singularities and $[H_\m^i(R)]_{>0}=0$ for every $i\geq0$.
\end{enumerate}
\end{theorem}
\begin{proof}
If $R$ is Du Bois, so is $R_P$ for all homogeneous primes $P$. So both $(1)$ and $(2)$ imply $Z$ has Du Bois singularities. So without loss of generality we assume $Z$ is Du Bois. Since $\mathscr{L}$ is ample, we know that $R$ is a finitely generated $K$-algebra. We pick homogeneous elements $x_1,\dots,x_m$ in $R$ that form a set of algebra generators of $R$ over $K$. Let $x_j$ have degree $d_j>0$. We have a natural degree-preserving map $S=K[x_1,\dots,x_m]\twoheadrightarrow R$ where $S$ is the polynomial ring with a possibly non-standard grading. Let $X=\Spec R$, $Y=\Spec S$, we have $X\hookrightarrow Y\cong\mathbb{A}^m$.

Now we use $R^\natural$ and $S^\natural$ to denote the Rees algebra of $R$ and $S$ with respect to the natural filtration $R_{\geq t}$ and $S_{\geq t}$. That is, \[R^\natural=R\oplus R_{\geq 1}\oplus R_{\geq 2}\oplus\cdots,\] \[S^{\natural}=S\oplus S_{\geq 1}\oplus S_{\geq 2}\oplus\cdots.\] Let $\widetilde{X}=\Proj R^\natural$ and $\widetilde{Y}=\Proj S^\natural$. We have natural maps $\widetilde{X}\to X$ and $\widetilde{Y}\to Y$ induced by the inclusion $R\hookrightarrow R^\natural$ and $S\hookrightarrow S^\natural$. We note that when $x_1,\dots,x_m$ all have degree 1, $\widetilde{X}$ and $\widetilde{Y}$ are just the blow ups of $X$ and $Y$ at the homogeneous maximal ideals.

It is straightforward to check that the reduced pre-image of $X$ in $\widetilde{Y}$ is $\overline{X}=\widetilde{X}\bigcup E$ where $E\cong\Proj(S/S_{\geq 1}\oplus S_{\geq 1}/S_{\geq2}\oplus\cdots)\cong \Proj S$ is a weighted projective space. It is also clear that $\widetilde{X}\bigcap E\cong \Proj R=Z$.

We further let $Y_0\xrightarrow{f}\widetilde{Y}$ be a strong log resolution of the pair $(\widetilde{Y}, \overline{X})$ and we use $X_0$ to denote the reduced pre-image of $\overline{X}$ in $Y_0$. We summarize all the above information in the following diagram:
\[  \xymatrix{
    X_0 \ar@{^{(}->}[r]  \ar[d]_{f} &  Y_0   \ar[d]_{f} \ar@/^1pc/[dd]^{\pi}  \\
    \widetilde{X}\bigcup E=\overline{X}\ar@{^{(}->}[r]  \ar[d]_g &  \widetilde{Y} \ar[d]_g  \\
    X\ar@{^{(}->}[r] & Y 
} \]

First notice that we can also interpret $\widetilde{X}=\Proj R^\natural$ as the total space of the tautological line bundle $\mathscr{L}^{-1}$ on $Z$ (see Section 8.7.3 in \cite{EGA}). Hence $\widetilde{X}$ has Du Bois singularities since $Z$ has Du Bois singularities by assumption. Also notice that $E$ is isomorphic to a weighted projective space, so it has rational singularities and hence is Du Bois. Now since $\widetilde{X}$, $E$, and $\widetilde{X}\bigcap E\cong Z$ are all Du Bois, it follows that $\overline{X}$ is Du Bois because Du Bois singularities glue well (for example, see 3.8, 4.10 in \cite{DuBoisMain} or Lemma 3.4 in \cite{SchwedeEasyCharacterization}).

Second notice that we have a spectral sequence \[R^pg_*R^qf_*O_{X_0}\Rightarrow R^{p+q}\pi_*O_{X_0}.\] Since $Y_0\rightarrow \widetilde{Y}$ is a strong log resolution of $(\widetilde{Y}, \overline{X})$ and $\overline{X}$ is Du Bois, we know that $f_*O_{X_0}=O_{\overline{X}}$ and $R^{q}f_*O_{X_0}=0$ for $q>0$ by Definition \ref{Schwede's simple characterization of Du Bois}. So the above spectral sequence degenerates. We have
\begin{equation}\label{1}
R^i\pi_*O_{X_0}\cong R^ig_*O_{\overline{X}}.
\end{equation}

Now we compute $R^ig_*O_{\overline{X}}$. Since $X$ is affine, this is just $H^i(\overline{X}, O_{\overline{X}})$. Since $\overline{X}=\widetilde{X}\bigcup E$, $Z\cong \widetilde{X}\bigcap E$, we have an exact sequence \[0\rightarrow O_{\overline{X}}\rightarrow O_{\widetilde{X}}\oplus O_E\rightarrow O_Z\rightarrow 0.\] This induces a long exact sequence on cohomology:
\begin{equation}
\label{7}\cdots \rightarrow H^i(\overline{X}, O_{\overline{X}})\rightarrow H^i(\widetilde{X}, O_{\widetilde{X}})\oplus H^i(E, O_E)\rightarrow H^i(Z,O_Z)\rightarrow \cdots.
\end{equation}
We know that $H^i(E, O_E)=0$ for every $i\geq 1$ because $E$ is a weighted projective space. We also know $H^i(Z, O_Z)\cong [H^{i+1}_\m(R)]_0$ for every $i\geq 1$ because $Z\cong \Proj R$. Now we use \v{C}ech complex to understand the map $H^i(\widetilde{X}, O_{\widetilde{X}})\rightarrow H^i(Z,O_Z)$. Recall that $x_1,\dots,x_m$ are homogeneous algebra generators of $R$ over $K$ of degree $d_1,\dots,d_m$. The natural map $O_{\widetilde{X}}\rightarrow O_Z$ induces a map between the $s$-th spot of the \v{C}ech complexes of $O_{\widetilde{X}}$ and $O_Z$ with respect to the affine cover $\{D_+(x_i)\}_{1\leq i\leq m}$. This induced map on \v{C}ech complexes can be explicitly described as follows (all the direct sum in the following diagram is taking over all $s$-tuples $1\leq i_1<\cdots <i_s\leq m$):
\[  \xymatrix{
\oplus O_{\widetilde{X}}(D_+(x_{i_1}x_{i_2}\cdots x_{i_s}))  \ar[r] \ar[d]^{\cong} & \oplus O_Z(D_+(x_{i_1}x_{i_2}\cdots x_{i_s}))\ar[d]^{\cong}  \\
\D\oplus\{\frac{y}{(x_{i_1}x_{i_2}\cdots x_{i_s})^n}| n\geq 0, y\in R_{\geq nd}, d=d_{i_1}+\cdots+ d_{i_s}\} \ar[r] \ar[d]^{\cong} & \oplus[R_{x_{i_1}x_{i_2}\cdots x_{i_s}}]_0 \ar[d]^\cong \\
\oplus[R_{x_{i_1}x_{i_2}\cdots x_{i_s}}]_{\geq 0} \ar[r]^\phi & \oplus[R_{x_{i_1}x_{i_2}\cdots x_{i_s}}]_0   
} \]

It is straightforward to check that the induced map on the second line takes the element $\D\frac{y}{(x_{i_1}x_{i_2}\cdots x_{i_s})^n}$ to $\D\frac{\overline{y}}{(x_{i_1}x_{i_2}\cdots x_{i_s})^n}$ where $\overline{y}$ denotes the image of $y$ in $R_{\geq nd}/R_{\geq nd+1}$. Hence the same map $\phi$ on the third line is exactly ``taking the degree $0$ part". By the \v{C}ech complex computation of sheaf cohomology, we know that $H^i(\widetilde{X}, O_{\widetilde{X}})\cong [H^{i+1}_\m(R)]_{\geq 0}$ and the map $H^i(\widetilde{X}, O_{\widetilde{X}})\rightarrow H^i(Z, O_Z)$ is exactly the natural map $[H^{i+1}_\m(R)]_{\geq 0}\rightarrow [H^{i+1}_\m(R)]_{0}$ for every $i\geq 1$. Hence the above long exact sequence (\ref{7}) gives (for every $i\geq 1$)
\begin{equation}
\label{2}
H^i(\overline{X}, O_{\overline{X}})=[H_\m^{i+1}(R)]_{>0}.
\end{equation}
From (\ref{1}) and (\ref{2}) it is clear that, for every $i\geq 1$, we have
\begin{equation}\label{4}
R^i\pi_*O_{X_0}\cong[H_\m^{i+1}(R)]_{>0}.
\end{equation}

However, it is straightforward to check that $Y_0\xrightarrow{\pi}Y$ is also a strong log resolution of the pair $(Y, X)$ and $X_0$ is the reduced pre-image of $X$ in $Y_0$. Hence by Definition \ref{Schwede's simple characterization of Du Bois}, we know that $X$ has Du Bois singularities if and only if $\pi_*O_{X_0}=O_X$ and $R^i\pi_*O_{X_0}=0$ for $i\geq 1$. But since $R$ is a section ring of a normal variety, $R$ is normal. So $\pi_*O_{X_0}=O_X$ is always true by Corollary 5.7 in \cite{SchwedeFInjectiveAreDuBois}. So by (\ref{4}), $X=\Spec R$ has Du Bois singularities if and only if $[H_\m^i(R)]_{>0}=0$ for $i\geq 0$ (note that $H_\m^0(R)$ and $H_\m^1(R)$ vanish because $R$ is normal).
\end{proof}

We state and prove our main theorem in characteristic $0$.

\begin{theorem}
Let $Z$ be a normal projective variety over $K$ and let $\mathscr{L}$ be an ample line bundle on $Z$. Let $R=\oplus H^0(Z,\mathscr{L}^i)$ be the section ring of $Z$ with respect to $\mathscr{L}$ and $\m$ be the irrelevant maximal ideal of $R$. If $Z$ is Cohen-Macaulay and $R$ has Du Bois singularities, then $R$ is Buchsbaum.

In particular, when $(R,\m)$ is a normal standard graded $K$-algebra, if $R$ is Du Bois and $R$ has finite local cohomology, then $R$ is Buchsbaum.
\end{theorem}
\begin{proof}
The last assertion is clear because every normal standard graded $K$-algebra is the section ring of a normal projective variety $Z\cong\Proj R$ with respect to some very ample line bundle $\mathscr{L}$, and the fact that $R$ has finite local cohomology implies $R$ is Cohen-Macaulay on the punctured spectrum, so it implies $Z$ is Cohen-Macaulay. Therefore it suffices to prove the first assertion.

Following Theorem \ref{sufficient condision for Buchsbaum ring in the graded case}, we will show that $[H_\m^i(R)]_s=0$ for every $s\neq 0$ and $i<n=\dim R$. Since $R$ is normal, $H_\m^i(R)=0$ for $i=0,1$ so there's nothing prove in these cases. Now for $2\leq i<n$, we will show $[H_\m^i(R)]_s=0$ for $s>0$ and $s<0$.

For $s>0$, this follows immediately from Theorem \ref{characterization of Du Bois in graded case} since $R$ has Du Bois singularities. For $s<0$, we want to show $[H_\m^{i+1}(R)]_s=H^i(Z,\mathscr{L}^s)=0$ for every $i<n-1$. Since $Z=\Proj R$ is Cohen-Macaulay and Du Bois of dimension $n-1$, by Serre duality, \[H^i(Z, \mathscr{L}^s)\cong H^{n-1-i}(Z, \omega_Z\otimes\mathscr{L}^{-s})=0\] where the last equality is by Theorem \ref{Kodaira vanishing for CM Du Bois} since $n-1-i>0$.
\end{proof}

In \cite{SchwedeFInjectiveAreDuBois}, the following conjecture was made:
\begin{conjecture}[{\it cf.} Question 8.1 in \cite{SchwedeFInjectiveAreDuBois} or Conjecture 4.1 in \cite{BhattSchwedeTakagiweakordinaryconjectureandFsingularity}]
\label{conjecture DB=F-injectivetype}
Let $X$ be a reduced scheme of finite type over an algebraically closed field of characteristic 0. Then $X$ has Du Bois singularities if and only if $X$ is of dense $F$-injective type.
\end{conjecture}

We note that the ``if" direction of the above conjecture is true by the main result in \cite{SchwedeFInjectiveAreDuBois}. And by a recent result of Bhatt, Schwede and Takagi \cite{BhattSchwedeTakagiweakordinaryconjectureandFsingularity}, this conjecture is equivalent to the weak ordinarity conjecture of Musta\c{t}\u{a} and Srinivas \cite{MustataSrinivasweakordinaryconjecture}. The point we want to make here is that under mild conditions on $R$, a positive answer to Conjecture \ref{conjecture DB=F-injectivetype} will give a positive answer to Question \ref{Takagi's question in char 0} by the standard method of reduction mod $p$.

\begin{proposition}
Suppose Conjecture \ref{conjecture DB=F-injectivetype} holds. Let $R$ be a ring finitely generated over an algebraically closed field $K$ of characteristic $0$. Suppose $R$ is equidimensional and $\m$ is the only non-Cohen-Macaulay (closed) point of $\Spec R$. If $R$ is Du Bois, then $R_\m$ is Buchsbaum.
\end{proposition}
\begin{proof}
Let $R$ be generated over $K$ by $z_1,\dots,z_k$, i.e., $R=K[z_1,\dots,z_k]/J$. Since $R$ is an affine algebra over an algebraically closed field $K$, without loss of generality we may assume $\m=(z_1,\dots,z_k)$ is the isolated non-Cohen-Macaulay point. If $R_\m$ is not Buchsbaum, then there exists system of parameters $\underline{x}=x_1,\dots,x_n$, $\underline{y}=y_1,\dots,y_n$ of $R_\m$ such that \[l(R_\m/(\underline{x}))-e(\underline{x}, R_\m)\neq l(R_\m/(\underline{y}))-e(\underline{x}, R_\m).\]
And we can certainly assume $\underline{x}$, $\underline{y}$ are actually elements of $R$. We note that $l(R_\m/(\underline{x}))-e(\underline{x}, R_\m)=\chi_1(\underline{x}, R_\m)=\sum_{i=1}^n (-1)^il(H_i(\underline{x}, R)_\m)$. Obviously, the support of all $H_i(\underline{x}, R)$, $H_i(\underline{y}, R)$ consist only of finitely many maximal ideals $\m=\m_0, \m_1,\dots,\m_t$ of $R$. Now we pick $f\notin \m$ but $f\in \m_i$ for every $i\geq 1$ and we localize at $f$. Notice that all the hypothesis on $R$ are unchanged. But now each $H_i(\underline{x}, R)$, $H_i(\underline{y}, R)$ is only supported at $\m$. In particular we know that
\begin{equation*}
\sum_{i=1}^n(-1)^i l(H_i(\underline{x}, R))=l(R_\m/(\underline{x}))-e(\underline{x}, R_\m)\neq l(R_\m/(\underline{y}))-e(\underline{x}, R_\m)= \sum_{i=1}^n (-1)^il(H_i(\underline{y}, R)).
\end{equation*}
Now we take a finitely generated $\mathbb{Z}$-algebra $A\subseteq K$ such that $R$ is well-defined over $A$ and let $R_A$ be the corresponding subring of $R$ so that $R=R_A\otimes_AK$. By generic freeness, we can shrink $A$ and make all the kernels, cokernels and homologies in the Koszul complex \[0\rightarrow R_A\to R_A^n\to\cdots\to R_A^n\xrightarrow{\underline{x}} R_A\to 0\] to be free as $A$-modules (and similar to $\underline{y}$). In particular, we have $H_i(\underline{x}, R)=H_i(\underline{x}, R_A)\otimes_AK$, $H_i(\underline{y}, R)=H_i(\underline{y}, R_A)\otimes_AK$ and all $H_i(\underline{x}, R_A)$, $H_i(\underline{y}, R_A)$ are free $A$-modules. In particular, we know that $l(H_i(\underline{x}, R))=\rank_A(H_i(\underline{x}, R_A))$ and $l(H_i(\underline{y}, R))=\rank_A(H_i(\underline{y}, R_A))$. Hence we have
\begin{eqnarray*}
&&\sum_{i=1}^n(-1)^i\rank_A(H_i(\underline{x}, R_A))=\sum_{i=1}^n(-1)^i l(H_i(\underline{x}, R))\\
&\neq& \sum_{i=1}^n(-1)^i l(H_i(\underline{y}, R))=\sum_{i=1}^n(-1)^i\rank_A(H_i(\underline{y}, R_A)).
\end{eqnarray*}
Now we pass to closed point of $A$, i.e., tensoring with $A/s$ for $s$ a maximal ideal of $A$. Let $R_s=R_A\otimes_AA/s$ and let $\m_s$ be the maximal ideal of $R_s$ corresponds to $\m$. Notice that when we pass to $R_s$, the above will give us \[\sum_{i=1}^n(-1)^i l(H_i(\underline{x}, R_s))\neq \sum_{i=1}^n(-1)^i l(H_i(\underline{y}, R_s))\] because $H_i(\underline{x}, R_A)$ and $H_i(\underline{y}, R_A)$ are free over $A$. In particular, this is saying that $(R_s)_{\m_s}$ is {\it not} Buchsbaum.

On the other hand, we can shrink $A$ to make $A$ regular and each $R_A[\frac{1}{z_i}]$ free over $A$ by generic freeness. Let $L$ be the fraction field of $A$. We know that \[L\otimes_AR_A[\frac{1}{z_i}]\to K\otimes_AR_A[\frac{1}{z_i}]=R[\frac{1}{z_i}]\] is faithfully flat with $0$-dimensional fiber. Since $(z_1,\dots,z_k)$ is the isolated non-Cohen-Macaulay point, each $R[\frac{1}{z_i}]$ is Cohen-Macaulay. It follows that each $L\otimes_AR_A[\frac{1}{z_i}]$ is Cohen-Macaulay. So we can further shrink $A$ such that each $R_A[\frac{1}{z_i}]$ is Cohen-Macaulay. Now since $A$ is regular and $A\to R_A[\frac{1}{z_i}]$ is faithfully flat, it is easy to see that after tensoring with $A/s$ for each maximal ideal $s$ of $A$, the resulting $R_s[\frac{1}{z_i}]$ is still Cohen-Macaulay for each $i$. Hence we may assume that after we pass to each closed point of $A$ (i.e., after we do the mod $p$ reduction), each $R_s$ still has an isolated non-Cohen-Macaulay point $\m_s$. And a similar argument shows that we may also assume each $R_s$ is equidimensional. Now if Conjecture \ref{conjecture DB=F-injectivetype} is true, then there should be a dense set of $s\in \Spec A$ such that $R_s$ is $F$-injective, equidimensional with an isolated non-Cohen-Macaulay point $\m_s$. So by Corollary \ref{main corollary 2}, $(R_s)_{\m_s}$ is Buchsbaum for a dense set of $s$, this is a contradiction.
\end{proof}

\section*{Acknowledgement}
I would like to thank Mel Hochster for many helpful and valuable discussions on this problem. I would like to thank Karl Schwede for answering many of my questions, for carefully reading a preliminary version this manuscript and for his comments. I am grateful to Kazuma Shimomoto for some helpful discussions on Buchsbaum rings, and to Zhixian Zhu for some discussions on Du Bois singularities. I would also like to thank the anonymous referee, whose comments helped improve the paper.

\bibliographystyle{skalpha}
\bibliography{CommonBib}

\end{document}